\documentclass{amsart}
\usepackage{amssymb}
\usepackage{amsfonts}

\newtheorem{theorem}{Theorem}
\theoremstyle{plain}

\newtheorem{definition}{Definition}

\newtheorem{lemma}{Lemma}

\newtheorem{remark}{Remark}

\numberwithin{equation}{section}
\input{tcilatex}

\begin{document}
\title[Two New Convex Dominated Functions]{Hermite-Hadamard-Type Inqualities
for New Different Kinds of Convex Dominated Functions}
\author{M.Emin \"{O}zdemir$^{\blacktriangle }$}
\address{$^{\blacktriangle }$Atat\"{u}rk University, K.K. Education Faculty,
Department of Mathematics, 25240, Campus, Erzurum, Turkey}
\email{emos@atauni.edu.tr}
\author{Havva Kavurmac\i $^{\blacktriangle ,\blacksquare }$}
\email{hkavurmaci@atauni.edu.tr}
\thanks{$^{\blacksquare }$Corresponding Author}
\author{Mevl\"{u}t Tun\c{c}}
\address{Department of Mathematics, Faculty of Art and Sciences, Kilis 7
Aralik University, Kilis, 79000, Turkey}
\email{mevluttunc@kilis.edu.tr}
\date{February 2, 2012}
\subjclass[2000]{ Primary 26D15, Secondary 26D10, 05C38}
\keywords{$m-$convex dominated function, Hermite-Hadamard's inequality, $%
\left( \alpha ,m\right) -$convex function, $r-$convex function.}

\begin{abstract}
In this paper, we establish several new convex dominated functions and then
we obtain new Hadamard type inequalities.
\end{abstract}

\maketitle

\section{Introduction}

The inequality%
\begin{equation}
f\left( \frac{a+b}{2}\right) \leq \frac{1}{b-a}\int_{a}^{b}f\left( x\right)
dx\leq \frac{f\left( a\right) +f\left( b\right) }{2}  \label{h}
\end{equation}%
which holds for all convex functions $f:[a,b]\rightarrow 
\mathbb{R}
$, is known in the literature as Hermite-Hadamard's inequality.

In \cite{G}, Toader defined $m-$convexity as the following:

\begin{definition}
The function $f:[0,b]\rightarrow 
\mathbb{R}
,$ $b>0$, is said to be $m-$convex where $m\in \lbrack 0,1],$ if we have%
\begin{equation*}
f(tx+m(1-t)y)\leq tf(x)+m(1-t)f(y)
\end{equation*}%
for all $x,y\in \lbrack 0,b]$ and $t\in \lbrack 0,1].$ We say that $f$ is $%
m- $concave if $\left( -f\right) $ is $m-$convex.
\end{definition}

In \cite{D}, Dragomir proved the following theorem.

Let $f:\left[ 0,\infty \right) \rightarrow 
\mathbb{R}
$ be an $m-$convex function with $m\in \left( 0,1\right] $ and $0\leq a<b.$
If $f\in L_{1}\left[ a,b\right] ,$ then the following inequalities hold:%
\begin{eqnarray}
f\left( \frac{a+b}{2}\right) &\leq &\frac{1}{b-a}\int_{a}^{b}\frac{f\left(
x\right) +mf\left( \frac{x}{m}\right) }{2}dx  \label{h1} \\
&\leq &\frac{1}{2}\left[ \frac{f\left( a\right) +mf\left( \dfrac{a}{m}%
\right) }{2}+m\frac{f\left( \frac{b}{m}\right) +mf\left( \dfrac{b}{m^{2}}%
\right) }{2}\right] .  \notag
\end{eqnarray}

In \cite{DI} and \cite{DPP}, the authors connect together some disparate
threads through a Hermite-Hadamard motif. The first of these threads is the
unifying concept of a $g-$convex dominated function. Similarly, in \cite{KOS}%
, Kavurmac\i\ et al. introduced the following class of functions and then
proved a theorem for this class of functions related to (\ref{h1}).

\begin{definition}
Let $g:\left[ 0,b\right] \rightarrow 
\mathbb{R}
$ be a given $m-$convex function on the interval $\left[ 0,b\right] $. The
real function $f:\left[ 0,b\right] \rightarrow 
\mathbb{R}
$ is called $\left( g,m\right) -$convex dominated on $\left[ 0,b\right] $ if
the following condition is satisfied%
\begin{eqnarray*}
&&\left\vert \lambda f(x)+m(1-\lambda )f(y)-f\left( \lambda x+m\left(
1-\lambda \right) y\right) \right\vert \\
&\leq &\lambda g(x)+m(1-\lambda )g(y)-g\left( \lambda x+m\left( 1-\lambda
\right) y\right)
\end{eqnarray*}%
for all $x,y\in \left[ 0,b\right] $, $\lambda \in \left[ 0,1\right] $ and $%
m\in \left[ 0,1\right] .$
\end{definition}

\begin{theorem}
\label{a} Let $g:\left[ 0,\infty \right) \rightarrow 
\mathbb{R}
$ be an $m-$convex function with $m\in \left( 0,1\right] $. $f:\left[
0,\infty \right) \rightarrow 
\mathbb{R}
$ is $\left( g,m\right) -$convex dominated \ mapping and $0\leq a<b.$ If $%
f\in L_{1}\left[ a,b\right] ,$ then one has the inequalities:%
\begin{eqnarray*}
&&\left\vert \frac{1}{b-a}\int_{a}^{b}\frac{f\left( x\right) +mf\left( \frac{%
x}{m}\right) }{2}dx-f\left( \frac{a+b}{2}\right) \right\vert \\
&& \\
&\leq &\frac{1}{b-a}\int_{a}^{b}\frac{g\left( x\right) +mg\left( \frac{x}{m}%
\right) }{2}dx-g\left( \frac{a+b}{2}\right)
\end{eqnarray*}%
and%
\begin{eqnarray*}
&&\left\vert \frac{1}{2}\left[ \frac{f\left( a\right) +mf\left( \dfrac{a}{m}%
\right) }{2}+m\frac{f\left( \frac{b}{m}\right) +mf\left( \dfrac{b}{m^{2}}%
\right) }{2}\right] -\frac{1}{b-a}\int_{a}^{b}\frac{f\left( x\right)
+mf\left( \frac{x}{m}\right) }{2}dx\right\vert \\
&& \\
&\leq &\frac{1}{2}\left[ \frac{g\left( a\right) +mg\left( \dfrac{a}{m}%
\right) }{2}+m\frac{g\left( \frac{b}{m}\right) +mg\left( \dfrac{b}{m^{2}}%
\right) }{2}\right] -\frac{1}{b-a}\int_{a}^{b}\frac{g\left( x\right)
+mg\left( \frac{x}{m}\right) }{2}dx.
\end{eqnarray*}
\end{theorem}

In \cite{MIH}, definition of $\left( \alpha ,m\right) -$convexity was
introduced by Mihe\c{s}an as the following.

\begin{definition}
\label{d1} The function $f:[0,b]\rightarrow 
\mathbb{R}
,$ $b>0$, is said to be $\left( \alpha ,m\right) -$convex, where $\left(
\alpha ,m\right) \in \lbrack 0,1]^{2},$ if we have%
\begin{equation*}
f(tx+m(1-t)y)\leq t^{\alpha }f(x)+m(1-t^{\alpha })f(y)
\end{equation*}%
for all $x,y\in \lbrack 0,b]$ and $t\in \lbrack 0,1].$
\end{definition}

Denote by $K_{m}^{\alpha }\left( b\right) $ the class of all $\left( \alpha
,m\right) -$convex functions on $\left[ 0,b\right] $ for which $f\left(
0\right) \leq 0.$ If we take $\left( \alpha ,m\right) =\left\{ \left(
0,0\right) ,\left( \alpha ,0\right) ,\left( 1,0\right) ,\left( 1,m\right)
,\left( 1,1\right) ,\left( \alpha ,1\right) \right\} ,$ it can be easily
seen that $\left( \alpha ,m\right) -$convexity reduces to increasing: $%
\alpha -$starshaped, starshaped, $m-$convex, convex and $\alpha -$convex,
respectively.

In \cite{SSOR}, Set et al. proved the following Hadamard type inequalities
for $\left( \alpha ,m\right) -$convex functions.

\begin{theorem}
Let $f:\left[ 0,\infty \right) \rightarrow 
\mathbb{R}
$ be an $\left( \alpha ,m\right) -$convex function with $\left( \alpha
,m\right) \in \left( 0,1\right] ^{2}.$ If \ $0\leq a<b<\infty $ and $f\in $ $%
L_{1}\left[ a,b\right] \cap L_{1}\left[ \frac{a}{m},\frac{b}{m}\right] ,$
then the following inequality holds:%
\begin{equation}
f\left( \frac{a+b}{2}\right) \leq \frac{1}{b-a}\int_{a}^{b}\frac{f\left(
x\right) +m\left( 2^{\alpha }-1\right) f\left( \frac{x}{m}\right) }{%
2^{\alpha }}dx.  \label{h2}
\end{equation}
\end{theorem}

\begin{theorem}
Let $f:\left[ 0,\infty \right) \rightarrow 
\mathbb{R}
$ be an $\left( \alpha ,m\right) -$convex function with $\left( \alpha
,m\right) \in \left( 0,1\right] ^{2}.$ If \ $0\leq a<b<\infty $ and $f\in $ $%
L_{1}\left[ a,b\right] ,$ then the following inequality holds:%
\begin{equation}
\frac{1}{b-a}\int_{a}^{b}f\left( x\right) dx\leq \frac{1}{2}\left[ \frac{%
f\left( a\right) +f\left( b\right) +m\alpha f\left( \frac{a}{m}\right)
+m\alpha f\left( \frac{b}{m}\right) }{\alpha +1}\right] .  \label{h3}
\end{equation}
\end{theorem}

For the recent results based on the above definition see the papers \cite%
{BOP}, \cite{BPR}, \cite{OKS}, \cite{OAK} and \cite{SSO}.

In \cite{OSA}, the power mean $M_{r}(x,y;\lambda )$ of order $r$ of positive
numbers $x,y$ is defined by%
\begin{equation*}
M_{r}(x,y;\lambda )=\left\{ 
\begin{array}{cc}
\left( \lambda x^{r}+\left( 1-\lambda \right) y^{r}\right) ^{\frac{1}{r}}, & 
r\neq 0 \\ 
x^{\lambda }y^{1-\lambda }, & r=0.%
\end{array}%
\right. 
\end{equation*}%
A positive function $f$ is $r-$convex on $[a,b]$ if for all $x,y\in \lbrack
a,b]$ and $\lambda \in \lbrack 0,1]$%
\begin{equation}
f(\lambda x+(1-\lambda )y)\leq M_{r}(f\left( x\right) ,f\left( y\right)
;\lambda ).  \label{h4}
\end{equation}%
The generalized logarithmic mean of order $r$ of positive numbers $x,y$ is
defined by%
\begin{equation}
L_{r}\left( x,y\right) =\left\{ 
\begin{array}{cc}
\frac{r}{r+1}\frac{x^{r+1}-y^{r+1}}{x^{r}-y^{r}}, & r\neq 0,1,x\neq y \\ 
&  \\ 
\frac{x-y}{\ln x-\ln y}, & r=0,x\neq y \\ 
&  \\ 
xy\frac{\ln x-\ln y}{x-y}, & r=-1,x\neq y \\ 
&  \\ 
x, & x=y%
\end{array}%
\right.   \label{L}
\end{equation}%
In \cite{GPP}, the following theorem was proved by Gill et al. for $r-$%
convex functions.

\begin{theorem}
Suppose $f$ is a positive $r-$convex function on $\left[ a,b\right] .$ Then 
\begin{equation}
\frac{1}{b-a}\int_{a}^{b}f\left( x\right) dx\leq L_{r}\left( f\left(
a\right) ,f\left( b\right) \right) .  \label{h5}
\end{equation}%
If $f$ is a positive $r-$concave function, then the inequality is reversed.
\end{theorem}

In the following sections our main results are given: We establish several
new convex dominated functions and then we obtain new Hadamard type
inequalities.

\section{$\left( g-\left( \protect\alpha ,m\right) \right) $-convex
dominated functions}

\begin{definition}
\label{d2} Let $g:\left[ 0,b\right] \rightarrow 
\mathbb{R}
,$ $b>0$ be a given $\left( \alpha ,m\right) $-convex function on the
interval $\left[ 0,b\right] $. The real function $f:\left[ 0,b\right]
\rightarrow 
\mathbb{R}
$ is called $\left( g-\left( \alpha ,m\right) \right) $-convex dominated on $%
\left[ 0,b\right] $ if the following condition is satisfied%
\begin{eqnarray}
&&\left\vert \lambda ^{\alpha }f(x)+m(1-\lambda ^{\alpha })f(y)-f\left(
\lambda x+m\left( 1-\lambda \right) y\right) \right\vert  \label{h6} \\
&\leq &\lambda ^{\alpha }g(x)+m(1-\lambda ^{\alpha })g(y)-g\left( \lambda
x+m\left( 1-\lambda \right) y\right)  \notag
\end{eqnarray}%
for all $x,y\in \left[ 0,b\right] $, $\lambda \in \left[ 0,1\right] $ and $%
\left( \alpha ,m\right) \in \left[ 0,1\right] ^{2}.$
\end{definition}

The next simple characterisation of $\left( \alpha ,m\right) $-convex
dominated functions holds.

\begin{lemma}
\label{l1} Let $g:\left[ 0,b\right] \rightarrow 
\mathbb{R}
$ be an $\left( \alpha ,m\right) $-convex function on the interval $\left[
0,b\right] $ and the function $f:\left[ 0,b\right] \rightarrow 
\mathbb{R}
.$ The following statements are equivalent:
\end{lemma}

\begin{enumerate}
\item $f$ is $\left( g-\left( \alpha ,m\right) \right) $-convex dominated on 
$\left[ 0,b\right] .$

\item The mappings $g-f$ and $g+f$ are $\left( \alpha ,m\right) $-convex
functions on $\left[ 0,b\right] .$

\item There exist two $\left( \alpha ,m\right) $-convex mappings $h,k$
defined on $\left[ 0,b\right] $ such that%
\begin{equation*}
\begin{array}{ccc}
f=\frac{1}{2}\left( h-k\right) & \text{and} & g=\frac{1}{2}\left( h+k\right)%
\end{array}%
.
\end{equation*}
\end{enumerate}

\begin{proof}
1$\Longleftrightarrow $2 The condition (\ref{h6}) is equivalent to%
\begin{eqnarray*}
&&g\left( \lambda x+m\left( 1-\lambda \right) y\right) -\lambda ^{\alpha
}g(x)-m(1-\lambda ^{\alpha })g(y) \\
&\leq &\lambda ^{\alpha }f(x)+m(1-\lambda ^{\alpha })f(y)-f\left( \lambda
x+m\left( 1-\lambda \right) y\right)  \\
&\leq &\lambda ^{\alpha }g(x)+m(1-\lambda ^{\alpha })g(y)-g\left( \lambda
x+m\left( 1-\lambda \right) y\right) 
\end{eqnarray*}%
for all $x,y\in I$, $\lambda \in \left[ 0,1\right] $ and $\left( \alpha
,m\right) \in \left[ 0,1\right] ^{2}.$ The two inequalities may be
rearranged as%
\begin{equation*}
\left( g+f\right) \left( \lambda x+m\left( 1-\lambda \right) y\right) \leq
\lambda ^{\alpha }\left( g+f\right) (x)+m(1-\lambda ^{\alpha })\left(
g+f\right) (y)
\end{equation*}%
and%
\begin{equation*}
\left( g-f\right) \left( \lambda x+m\left( 1-\lambda \right) y\right) \leq
\lambda ^{\alpha }\left( g-f\right) (x)+m(1-\lambda ^{\alpha })\left(
g-f\right) (y)
\end{equation*}%
which are equivalent to the $\left( \alpha ,m\right) $-convexity of $g+f$
and $g-f,$ respectively.

2$\Longleftrightarrow $3 We define the mappings $f,g$ as $f=\frac{1}{2}%
\left( h-k\right) $ and $g=\frac{1}{2}\left( h+k\right) $. Then, if we sum
and subtract $f,g,$ respectively, we have $g+f=h$ and $g-f=k.$ By the
condition 2 of Lemma 1, the mappings $g-f$ and $g+f$ are $\left( \alpha
,m\right) $-convex on $\left[ 0,b\right] ,$ so $h,k$ are $\left( \alpha
,m\right) $-convex mappings too.
\end{proof}

\begin{theorem}
\label{t1} Let $g:\left[ 0,\infty \right) \rightarrow 
\mathbb{R}
$ be an $\left( \alpha ,m\right) -$convex function with $\left( \alpha
,m\right) \in \left( 0,1\right] ^{2}$. $f:\left[ 0,\infty \right)
\rightarrow 
\mathbb{R}
$ is $\left( g-\left( \alpha ,m\right) \right) -$convex dominated mapping
and $0\leq a<b.$ If $f\in L_{1}\left[ a,b\right] \cap L_{1}\left[ \frac{a}{m}%
,\frac{b}{m}\right] ,$ then the first inequality holds:%
\begin{eqnarray*}
&&\left\vert \frac{1}{b-a}\int_{a}^{b}\frac{f\left( x\right) +m\left(
2^{\alpha }-1\right) f\left( \frac{x}{m}\right) }{2^{\alpha }}dx-f\left( 
\frac{a+b}{2}\right) \right\vert \\
&& \\
&\leq &\frac{1}{b-a}\int_{a}^{b}\frac{g\left( x\right) +m\left( 2^{\alpha
}-1\right) g\left( \frac{x}{m}\right) }{2^{\alpha }}dx-g\left( \frac{a+b}{2}%
\right)
\end{eqnarray*}%
and if $f\in L_{1}\left[ a,b\right] $ then the second inequality holds: 
\begin{eqnarray*}
&&\left\vert \frac{1}{2}\left[ \frac{f\left( a\right) +mf\left( \dfrac{a}{m}%
\right) }{\alpha +1}+m\alpha \frac{f\left( \frac{b}{m}\right) +mf\left( 
\dfrac{b}{m^{2}}\right) }{\alpha +1}\right] -\frac{1}{b-a}\int_{a}^{b}\frac{%
f\left( x\right) +mf\left( \frac{x}{m}\right) }{2}dx\right\vert \\
&& \\
&\leq &\frac{1}{2}\left[ \frac{g\left( a\right) +mg\left( \dfrac{a}{m}%
\right) }{\alpha +1}+m\alpha \frac{g\left( \frac{b}{m}\right) +mg\left( 
\dfrac{b}{m^{2}}\right) }{\alpha +1}\right] -\frac{1}{b-a}\int_{a}^{b}\frac{%
g\left( x\right) +mg\left( \frac{x}{m}\right) }{2}dx.
\end{eqnarray*}
\end{theorem}

\begin{proof}
By Definition \ref{d2} with $\lambda =\frac{1}{2}$, as the mapping $f$ is $%
\left( g-\left( \alpha ,m\right) \right) -$convex dominated function, we
have that%
\begin{equation*}
\left\vert \frac{f\left( x\right) +m\left( 2^{\alpha }-1\right) f\left(
y\right) }{2^{\alpha }}-f\left( \frac{x+my}{2}\right) \right\vert \leq \frac{%
g\left( x\right) +m\left( 2^{\alpha }-1\right) g\left( y\right) }{2^{\alpha }%
}-g\left( \frac{x+my}{2}\right)
\end{equation*}%
for all $x,y\in \left[ 0,\infty \right) $ and $\left( \alpha ,m\right) \in
\left( 0,1\right] ^{2}.$ If we choose $x=ta+(1-t)b,$ $y=\left( 1-t\right) 
\frac{a}{m}+t\frac{b}{m}$ and $t\in \left[ 0,1\right] ,$ then we get%
\begin{eqnarray*}
&&\left\vert \frac{f\left( ta+(1-t)b\right) +m\left( 2^{\alpha }-1\right)
f\left( \frac{\left( 1-t\right) a+tb}{m}\right) }{2^{\alpha }}-f\left( \frac{%
a+b}{2}\right) \right\vert \\
&& \\
&\leq &\frac{g\left( ta+(1-t)b\right) +m\left( 2^{\alpha }-1\right) g\left( 
\frac{\left( 1-t\right) a+tb}{2}\right) }{2^{\alpha }}-g\left( \frac{a+b}{2}%
\right) .
\end{eqnarray*}%
Integrating over $t$ on $\left[ 0,1\right] $ we deduce that%
\begin{eqnarray*}
&&\left\vert \frac{\int_{0}^{1}f\left( ta+(1-t)b\right) dt+m\left( 2^{\alpha
}-1\right) \int_{0}^{1}f\left( \frac{\left( 1-t\right) a+tb}{m}\right) dt}{%
2^{\alpha }}-f\left( \frac{a+b}{2}\right) \right\vert \\
&& \\
&\leq &\frac{\int_{0}^{1}g\left( ta+(1-t)b\right) dt+m\left( 2^{\alpha
}-1\right) \int_{0}^{1}g\left( \frac{\left( 1-t\right) a+tb}{m}\right) dt}{%
2^{\alpha }}-g\left( \frac{a+b}{2}\right)
\end{eqnarray*}%
and so the first inequality is proved.

Since $f$ is $\left( g-\left( \alpha ,m\right) \right) -$convex dominated
function, we have%
\begin{eqnarray*}
&&\left\vert t^{\alpha }f\left( x\right) +m(1-t^{\alpha })f\left( y\right)
-f\left( tx+m(1-t)y\right) \right\vert \\
&& \\
&\leq &t^{\alpha }g\left( x\right) +m(1-t^{\alpha })g\left( y\right)
-g\left( tx+m(1-t)y\right) ,\text{ for all }x,y>0
\end{eqnarray*}%
which gives for $x=a$ and $y=\frac{b}{m}$%
\begin{eqnarray}
&&\left\vert t^{\alpha }f\left( a\right) +m(1-t^{\alpha })f\left( \frac{b}{m}%
\right) -f\left( ta+m(1-t)\frac{b}{m}\right) \right\vert  \label{h7} \\
&&  \notag \\
&\leq &t^{\alpha }g\left( a\right) +m(1-t^{\alpha })g\left( \frac{b}{m}%
\right) -g\left( ta+m(1-t)\frac{b}{m}\right)  \notag
\end{eqnarray}%
and for $x=\frac{a}{m}$, $y=\frac{b}{m^{2}}$ and then multiply with $m$ 
\begin{eqnarray}
&&\left\vert mtf\left( \frac{a}{m}\right) +m^{2}(1-t)f\left( \frac{b}{m^{2}}%
\right) -mf\left( t\frac{a}{m}+(1-t)\frac{b}{m}\right) \right\vert
\label{h8} \\
&&  \notag \\
&\leq &mtg\left( \frac{a}{m}\right) +m^{2}(1-t)g\left( \frac{b}{m^{2}}%
\right) -mg\left( t\frac{a}{m}+(1-t)\frac{b}{m}\right)  \notag
\end{eqnarray}%
for all $t\in \left[ 0,1\right] .$ By properties of modulus, if we add the
inequalities in $\left( \text{\ref{h7}}\right) $ and $\left( \text{\ref{h8}}%
\right) $, we get 
\begin{eqnarray*}
&&\left\vert t^{\alpha }\left[ f\left( a\right) +mf\left( \frac{a}{m}\right) %
\right] +m(1-t^{\alpha })\left[ f\left( \frac{b}{m}\right) +mf\left( \frac{b%
}{m^{2}}\right) \right] \right. \\
&& \\
&&-\left. \left[ f\left( ta+m(1-t)\frac{b}{m}\right) +mf\left( t\frac{a}{m}%
+(1-t)\frac{b}{m}\right) \right] \right\vert \\
&& \\
&\leq &t^{\alpha }\left[ g\left( a\right) +mg\left( \frac{a}{m}\right) %
\right] +m(1-t^{\alpha })\left[ g\left( \frac{b}{m}\right) +mg\left( \frac{b%
}{m^{2}}\right) \right] \\
&& \\
&&-\left[ g\left( ta+m(1-t)\frac{b}{m}\right) +mg\left( t\frac{a}{m}+(1-t)%
\frac{b}{m}\right) \right] .
\end{eqnarray*}%
Thus, integrating over $t$ on $\left[ 0,1\right] $ we obtain the second
inequality. The proof is completed.
\end{proof}

\begin{remark}
If we choose $\alpha =1$ in Theorem \ref{t1}, we get two inequalities of
Hermite-Hadamard type for functions that are $\left( g,m\right) -$convex
dominated in Theorem \ref{a}.
\end{remark}

\begin{theorem}
\label{t2} Let $g:\left[ 0,\infty \right) \rightarrow 
\mathbb{R}
$ be an $\left( \alpha ,m\right) -$convex function with $\left( \alpha
,m\right) \in \left( 0,1\right] ^{2}$. $f:\left[ 0,\infty \right)
\rightarrow 
\mathbb{R}
$ is $\left( g-\left( \alpha ,m\right) \right) -$convex dominated \ mapping
and $0\leq a<b.$ If $f\in L_{1}\left[ a,b\right] ,$ then the following
inequality holds:%
\begin{eqnarray}
&&\left\vert \frac{1}{2}\left[ \frac{f\left( a\right) +f\left( b\right)
+m\alpha f\left( \frac{a}{m}\right) +m\alpha f\left( \frac{b}{m}\right) }{%
\alpha +1}\right] -\frac{1}{b-a}\int_{a}^{b}f\left( x\right) dx\right\vert
\label{h9} \\
&&  \notag \\
&\leq &\frac{1}{2}\left[ \frac{g\left( a\right) +g\left( b\right) +m\alpha
g\left( \frac{a}{m}\right) +m\alpha g\left( \frac{b}{m}\right) }{\alpha +1}%
\right] -\frac{1}{b-a}\int_{a}^{b}g\left( x\right) dx  \notag
\end{eqnarray}
\end{theorem}

\begin{proof}
Since $f$ is $\left( g-\left( \alpha ,m\right) \right) -$convex dominated
function, we have%
\begin{eqnarray*}
&&\left\vert t^{\alpha }f\left( a\right) +m(1-t^{\alpha })f\left( \frac{b}{m}%
\right) -f\left( ta+m(1-t)\frac{b}{m}\right) \right\vert \\
&& \\
&\leq &t^{\alpha }g\left( a\right) +m(1-t^{\alpha })g\left( \frac{b}{m}%
\right) -g\left( ta+m(1-t)\frac{b}{m}\right)
\end{eqnarray*}%
and%
\begin{eqnarray*}
&&\left\vert t^{\alpha }f\left( b\right) +m(1-t^{\alpha })f\left( \frac{a}{m}%
\right) -f\left( tb+m(1-t)\frac{a}{m}\right) \right\vert \\
&& \\
&\leq &t^{\alpha }g\left( b\right) +m(1-t^{\alpha })g\left( \frac{a}{m}%
\right) -g\left( tb+m(1-t)\frac{a}{m}\right)
\end{eqnarray*}%
for all $t\in \left[ 0,1\right] .$ Adding the above inequalities, we get%
\begin{eqnarray*}
&&\left\vert t^{\alpha }\left[ f\left( a\right) +f\left( b\right) \right]
+m(1-t^{\alpha })\left[ f\left( \frac{a}{m}\right) +f\left( \frac{b}{m}%
\right) \right] -f\left( ta+m(1-t)\frac{b}{m}\right) -f\left( tb+m(1-t)\frac{%
a}{m}\right) \right\vert \\
&& \\
&\leq &t^{\alpha }\left[ g\left( a\right) +g\left( b\right) \right]
+m(1-t^{\alpha })\left[ g\left( \frac{a}{m}\right) +g\left( \frac{b}{m}%
\right) \right] -g\left( ta+m(1-t)\frac{b}{m}\right) -g\left( tb+m(1-t)\frac{%
a}{m}\right) .
\end{eqnarray*}%
Integrating over $t\in \left[ 0,1\right] $ and then by dividing the
resulting inequality with $2,$ we get the desired result. The proof is
completed.

Another proof can be done as the following.

Since $f$ is $\left( g-\left( \alpha ,m\right) \right) -$convex dominated,
we have by Lemma \ref{l1} that $g+f$ and $g-f$ are $\left( \alpha ,m\right)
- $convex on $\left[ 0,b\right] ,$ and so by the Hadamard's type inequality
for $\left( \alpha ,m\right) -$convex functions in $\left( \text{\ref{h3}}%
\right) $%
\begin{eqnarray}
&&\frac{1}{b-a}\int_{a}^{b}\left( g+f\right) \left( x\right) dx  \label{h10}
\\
&\leq &\frac{1}{2}\left[ \frac{\left( g+f\right) \left( a\right) +\left(
g+f\right) \left( b\right) +m\alpha \left( g+f\right) \left( \frac{a}{m}%
\right) +m\alpha \left( g+f\right) \left( \frac{b}{m}\right) }{\alpha +1}%
\right]  \notag
\end{eqnarray}%
and%
\begin{eqnarray}
&&\frac{1}{b-a}\int_{a}^{b}\left( g-f\right) \left( x\right) dx  \label{h11}
\\
&\leq &\frac{1}{2}\left[ \frac{\left( g-f\right) \left( a\right) +\left(
g-f\right) \left( b\right) +m\alpha \left( g-f\right) \left( \frac{a}{m}%
\right) +m\alpha \left( g-f\right) \left( \frac{b}{m}\right) }{\alpha +1}%
\right]  \notag
\end{eqnarray}%
By using the inequalities in $\left( \text{\ref{h10}}\right) $ and $\left( 
\text{\ref{h11}}\right) $, we get the inequality in $\left( \text{\ref{h9}}%
\right) $.
\end{proof}

\section{$\left( g,r\right) -$convex dominated functions}

\begin{definition}
\label{d3} Let positive function $g:\left[ a,b\right] \rightarrow 
\mathbb{R}
$ be a given $r-$convex function on $\left[ a,b\right] $. The real function $%
f:\left[ a,b\right] \rightarrow 
\mathbb{R}
$ is called $\left( g,r\right) -$convex dominated on $\left[ a,b\right] $ if
the following condition is satisfied:%
\begin{eqnarray*}
&&\left\vert M_{r}(f\left( x\right) ,f\left( y\right) ;\lambda )-f\left(
\lambda x+\left( 1-\lambda \right) y\right) \right\vert  \\
&\leq &M_{r}(g\left( x\right) ,g\left( y\right) ;\lambda )-g\left( \lambda
x+\left( 1-\lambda \right) y\right) 
\end{eqnarray*}%
for all $x,y\in \left[ a,b\right] $ and $\lambda \in \left[ 0,1\right] .$
\end{definition}

\begin{theorem}
Let positive function $g:\left[ a,b\right] \rightarrow 
\mathbb{R}
$ be an $r-$convex function on $\left[ a,b\right] $. $f:\left[ a,b\right]
\rightarrow 
\mathbb{R}
$ is $\left( g,r\right) -$convex dominated\ mapping and $0\leq a<b.$ If $%
f\in L_{1}\left[ a,b\right] ,$ then the following inequality holds:%
\begin{equation*}
\left\vert L_{r}\left( f\left( a\right) ,f\left( b\right) \right) -\frac{1}{%
b-a}\int_{a}^{b}f\left( x\right) dx\right\vert \leq L_{r}\left( g\left(
a\right) ,g\left( b\right) \right) -\frac{1}{b-a}\int_{a}^{b}g\left(
x\right) dx
\end{equation*}%
for all $x,y\in I$, $\lambda \in \left[ 0,1\right] $ and $L_{r}\left(
f\left( a\right) ,f\left( b\right) \right) $ as in (\ref{L}).
\end{theorem}

\begin{proof}
By the Definition \ref{d3} with $r=0,\ f\left( a\right) \neq f\left(
b\right) ,$ we have 
\begin{eqnarray*}
&&\left\vert f^{\lambda }\left( a\right) f^{1-\lambda }\left( b\right)
-f\left( \lambda a+\left( 1-\lambda \right) b\right) \right\vert \\
&\leq &g^{\lambda }\left( a\right) g^{1-\lambda }\left( b\right) -g\left(
\lambda a+\left( 1-\lambda \right) b\right) .
\end{eqnarray*}%
Integrating the above inequality over $\lambda $ on $\left[ 0,1\right] ,$ we
have 
\begin{eqnarray*}
&&\left\vert f\left( b\right) \int_{0}^{1}\left[ \frac{f\left( a\right) }{%
f\left( b\right) }\right] ^{\lambda }d\lambda -\int_{0}^{1}f\left( \lambda
a+\left( 1-\lambda \right) b\right) d\lambda \right\vert \\
&& \\
&\leq &g\left( b\right) \int_{0}^{1}\left[ \frac{g\left( a\right) }{g\left(
b\right) }\right] ^{\lambda }d\lambda -\int_{0}^{1}g\left( \lambda a+\left(
1-\lambda \right) b\right) d\lambda .
\end{eqnarray*}%
By a simple calculation we have%
\begin{eqnarray*}
&&\left\vert \frac{f\left( b\right) -f\left( a\right) }{\ln f\left( b\right)
-\ln f\left( a\right) }-\frac{1}{b-a}\int_{a}^{b}f\left( x\right)
dx\right\vert \\
&& \\
&\leq &\frac{g\left( b\right) -g\left( a\right) }{\ln g\left( b\right) -\ln
g\left( a\right) }-\frac{1}{b-a}\int_{a}^{b}g\left( x\right) dx.
\end{eqnarray*}%
The above inequality can written as%
\begin{equation*}
\left\vert L_{r}\left( f\left( a\right) ,f\left( b\right) \right) -\frac{1}{%
b-a}\int_{a}^{b}f\left( x\right) dx\right\vert \leq L_{r}\left( g\left(
a\right) ,g\left( b\right) \right) -\frac{1}{b-a}\int_{a}^{b}g\left(
x\right) dx.
\end{equation*}%
For $r=0,\ f\left( a\right) =f\left( b\right) ,$ we have with the same
development%
\begin{eqnarray*}
&&\left\vert f\left( a\right) -f\left( \lambda a+\left( 1-\lambda \right)
b\right) \right\vert \\
&\leq &g\left( a\right) -g\left( \lambda a+\left( 1-\lambda \right) b\right)
\end{eqnarray*}%
and this inequality can be written as%
\begin{equation*}
\left\vert L_{r}\left( f\left( a\right) ,f\left( b\right) \right) -\frac{1}{%
b-a}\int_{a}^{b}f\left( x\right) dx\right\vert \leq L_{r}\left( g\left(
a\right) ,g\left( b\right) \right) -\frac{1}{b-a}\int_{a}^{b}g\left(
x\right) dx.
\end{equation*}%
By the Definition \ref{d3} with $r\neq 0,-1,\ f\left( a\right) \neq f\left(
b\right) ,$ we have%
\begin{eqnarray*}
&&\left\vert \left( \lambda f^{r}\left( a\right) +\left( 1-\lambda \right)
f^{r}\left( b\right) \right) ^{\frac{1}{r}}-f\left( \lambda a+\left(
1-\lambda \right) b\right) \right\vert \\
&& \\
&\leq &\left( \lambda g^{r}\left( a\right) +\left( 1-\lambda \right)
g^{r}\left( b\right) \right) ^{\frac{1}{r}}-g\left( \lambda a+\left(
1-\lambda \right) b\right) .
\end{eqnarray*}%
Integrating the above inequality over $\lambda $ on $\left[ 0,1\right] ,$ we
have 
\begin{eqnarray*}
&&\left\vert \frac{r}{r+1}\frac{f^{r+1}\left( a\right) -f^{r+1}\left(
b\right) }{f^{r}\left( a\right) -f^{r}\left( b\right) }-\frac{1}{b-a}%
\int_{a}^{b}f\left( x\right) dx\right\vert \\
&& \\
&\leq &\frac{r}{r+1}\frac{g^{r+1}\left( a\right) -g^{r+1}\left( b\right) }{%
g^{r}\left( a\right) -g^{r}\left( b\right) }-\frac{1}{b-a}%
\int_{a}^{b}g\left( x\right) dx.
\end{eqnarray*}%
The above inequality can be written as%
\begin{eqnarray*}
&&\left\vert L_{r}\left( f\left( a\right) ,f\left( b\right) \right) -\frac{1%
}{b-a}\int_{a}^{b}f\left( x\right) dx\right\vert \\
&\leq &L_{r}\left( g\left( a\right) ,g\left( b\right) \right) -\frac{1}{b-a}%
\int_{a}^{b}g\left( x\right) dx.
\end{eqnarray*}%
For$\ r\neq 0$ and $f\left( a\right) =f\left( b\right) ,$ we have similarly%
\begin{eqnarray*}
&&\left\vert \left( f^{r}\left( a\right) \right) ^{\frac{1}{r}}-f\left(
\lambda a+\left( 1-\lambda \right) b\right) \right\vert \\
&\leq &\left( g^{r}\left( a\right) \right) ^{\frac{1}{r}}-g\left( \lambda
a+\left( 1-\lambda \right) b\right) .
\end{eqnarray*}%
Then integrating the above inequality over $\lambda $ on $\left[ 0,1\right]
, $ we have%
\begin{eqnarray*}
&&\left\vert L_{r}\left( f\left( a\right) ,f\left( b\right) \right) -\frac{1%
}{b-a}\int_{a}^{b}f\left( x\right) dx\right\vert \\
&\leq &L_{r}\left( g\left( a\right) ,g\left( b\right) \right) -\frac{1}{b-a}%
\int_{a}^{b}g\left( x\right) dx.
\end{eqnarray*}%
Finally, let $r=-1.$ For$\ f\left( a\right) \neq f\left( b\right) $ we have
again%
\begin{eqnarray*}
&&\left\vert \left( \lambda f^{-1}\left( a\right) +\left( 1-\lambda \right)
f^{-1}\left( b\right) \right) ^{-1}-f\left( \lambda a+\left( 1-\lambda
\right) b\right) \right\vert \\
&& \\
&\leq &\left( \lambda g^{-1}\left( a\right) +\left( 1-\lambda \right)
g^{-1}\left( b\right) \right) ^{-1}-g\left( \lambda a+\left( 1-\lambda
\right) b\right) .
\end{eqnarray*}%
Integrating the above inequality over $\lambda $ on $\left[ 0,1\right] ,$ we
have 
\begin{eqnarray*}
&&\left\vert \frac{f\left( a\right) f\left( b\right) }{f\left( b\right)
-f\left( a\right) }\int_{\frac{1}{f\left( a\right) }}^{\frac{1}{f\left(
b\right) }}\lambda ^{-1}d\lambda -\frac{1}{b-a}\int_{a}^{b}f\left( x\right)
dx\right\vert \\
&& \\
&\leq &\frac{g\left( a\right) g\left( b\right) }{g\left( b\right) -g\left(
a\right) }\int_{\frac{1}{f\left( a\right) }}^{\frac{1}{f\left( b\right) }%
}\lambda ^{-1}d\lambda -\frac{1}{b-a}\int_{a}^{b}g\left( x\right) dx.
\end{eqnarray*}%
The above inequality can be written as%
\begin{eqnarray*}
&&\left\vert L_{-1}\left( f\left( a\right) ,f\left( b\right) \right) -\frac{1%
}{b-a}\int_{a}^{b}f\left( x\right) dx\right\vert \\
&\leq &L_{-1}\left( g\left( a\right) ,g\left( b\right) \right) -\frac{1}{b-a}%
\int_{a}^{b}g\left( x\right) dx.
\end{eqnarray*}%
The proof is completed.
\end{proof}

\end{document}